\documentclass[preprint,12pt]{elsarticle}

\usepackage{amssymb}
\usepackage{amsthm}

\newcommand{\sysn}{\left\{\begin{array}{rcl}}
\newcommand{\sysk}{\end{array}\right.}

\newtheorem{theorem}{Theorem}[section]
\newtheorem{lemma}[theorem]{Lemma}

\theoremstyle{example}
\newtheorem{example}[theorem]{Example}

\newtheorem{proposition}[theorem]{Proposition}
\theoremstyle{definition}
\newtheorem{definition}[theorem]{Definition}
\newtheorem{remark}[theorem]{Remark}

\journal{...}

\begin{document}

\begin{frontmatter}



\title{The Menger and projective Menger properties of function spaces with the set-open topology}


\author{Alexander V. Osipov}


\ead{OAB@list.ru}


\address{Krasovskii Institute of Mathematics and Mechanics, Ural Federal
 University,

 Ural State University of Economics, 620219, Yekaterinburg, Russia}

\begin{abstract}
For a Tychonoff space $X$ and a family $\lambda$ of subsets of
$X$, we denote by $C_{\lambda}(X)$ the space of all real-valued
continuous functions on $X$ with the set-open topology. In this
paper, we study the Menger and projective Menger properties of a
Hausdorff space $C_{\lambda}(X)$. Our main results state that

if $\lambda$ is a $\pi$-network of $X$, then

 (1) $C_{\lambda}(X)$ is Menger space, if and only, if $C_{\lambda}(X)$ is
$\sigma$-compact,

and, if $Y$ is a dense subset of $X$, then

(2) $C_{p}(Y\vert X)$ is projective Menger space, if and only, if
$C_{p}(Y\vert X)$ is $\sigma$-pseudocompact.
\end{abstract}


\end{frontmatter}



\section{Introduction}
\label{}

Throughout this paper $X$ will be a Tychonoff space.  Let
$\lambda$ be a family non-empty subsets of $X$, $C(X)$ the set of
all continuous real-valued function on $X$. Denote by
$C_{\lambda}(X)$ the set $C(X)$ is endowed with the $\lambda$-open
topology.
 The elements of the standard subbases of the $\lambda$-open topology
 will be denoted as follows:
 $[F,\,U]=\{f\in C(X):\ f(F)\subseteq U\}$,  where $F\in\lambda$, $U$ is an open subset  of $\mathbb
 R$. Note that if $\lambda$ consists of all
finite subsets of $X$ then the $\lambda$-open topology is equal to
the topology of pointwise convergence, that is
$C_{\lambda}(X)=C_{p}(X)$. Denote be $C_{p}(Y\vert X)=\{h\in
C_p(Y): h=f\vert_Y$ for $f\in C(X)\}$ for $Y\subset X$.

Recall that, if $X$ is a space and $\mathcal{P}$  a topological
property, we say that $X$ is $\sigma$-$\mathcal{P}$ if $X$ is the
countable union of subspaces with the property $\mathcal{P}$.

So a space $X$ is called $\sigma$-compact ($\sigma$-pseudocompact,
$\sigma$-bounded), if $X=\bigcup\limits_{i=1}^{\infty}X_i$, where
$X_i$ is a compact (pseudocompact, bounded) for every $i\in
\mathbb N$. N.V. Velichko proved that $C_p(X)$ is
$\sigma$-compact, if and only, if $X$ is finite. In \cite{tka},
V.V. Tkachuk clarified when $C_p(X)$ is $\sigma$-pseudocompact and
when $C_p(X)$ is $\sigma$-bounded, and considered similar
questions for the space $C^*_p(X)$ of bounded continuous functions
on $X$.

A space $X$ is said to be Menger \cite{hur} (or, \cite{sash}) if
for every sequence $\{\mathcal{U}_n : n\in \omega\}$ of open
covers of $X$, there are finite subfamilies $\mathcal{V}_n\subset
\mathcal{U}_n$ such that $\bigcup \{\mathcal{V}_n : n\in \omega
\}$ is a cover of $X$.

 Every $\sigma$-compact space is Menger,
and a Menger space is Lindel$\ddot{o}$f. The Menger property is
closed hereditary, and it is preserved by continuous maps. It is
well known that the Baire space $\mathbb{N}^{\mathbb{N}}$ (hence,
$\mathbb{R}^{\omega}$) is not Menger.

In \cite{arh2}, A.V. Arhangel'skii proved that $C_p(X)$ is Menger,
if and only, if $X$ is finite.
\medskip

Let $\mathcal{P}$ be a topological property. A.V. Arhangel'skii
calls $X$ {\it projectively $\mathcal{P}$} if every second
countable image of $X$ is $\mathcal{P}$. Arhangel'skii consider
projective $\mathcal{P}$ for $\mathcal{P}=\sigma$-compact,
analytic \cite{arh}, and other properties.

Lj.D.R. Ko$\check{c}$inac characterized the classical covering
properties of Menger, Rothberger, Hurewicz and Gerlits-Nagy in
term of continuous images in $\mathbb{R}^{\omega}$. The projective
selection principles were introduced and first time considered in
\cite{koc}.

\medskip

Every Menger space is projectively Menger. It is known (Theorem
2.2 in \cite{koc}) that a space is Menger, if and only, if it is
Lindel$\ddot{o}$f and projectively Menger.

Characterizations of projectively Menger spaces $X$ in terms a
selection principle restricted to countable covers by cozero sets
are given in \cite{bcm}.

In \cite{sa}, M. Sakai proved that $C_p(X)$ is projectively
Menger, if and only, if $X$ is pseudocompact and $b$-discrete.

In this paper we study the Menger property of Hausdorff space
$C_\lambda(X)$, and the projective Menger property of
$C_{p}(Y\vert X)$ where $Y$ is dense subset of $X$.

\section{Main definitions and notation}

Recall that a family $\lambda$ of non-empty subsets of a
topological space $(X,\tau)$ is called a $\pi$-network for $X$ if
for any nonempty open set $U\in\tau$ there exists $A\in \lambda$
such that $A\subset U$.

Throughout this paper, a family $\lambda$ of nonempty subsets of
the set $X$ is
  a $\pi$-network. This condition is equivalent to the space $C_{\lambda}(X)$ being a Hausdorff space \cite{nokh}.

We will also need the following assertion \cite{arkh}, \cite{arp}.

\begin{proposition}\label{pr13} If $\mathbb{I}_{\alpha}=\mathbb{I}=[0,1]$ for
$\alpha\in A$ and $Y$ is a subspace of the Tychonoff cube
$\mathbb{I}^A=\prod \{\mathbb{I}_{\alpha}: \alpha\in A\}$ which,
whatever the countable set $B\subset A$, projects under the
canonical projection $\pi_B: \mathbb{I}^A \mapsto \mathbb{I}^B$
onto the whole cube $\mathbb{I}^B=\prod\{\mathbb{I}_{\alpha} :
\alpha\in B \}$ of $\mathbb{I}^A$, then $Y$ is pseudocompact.
\end{proposition}

\begin{theorem}(Nokhrin  \cite{nokh}) \label{th111} For a Tychonoff space $X$ the following statements are
equivalent:

\begin{enumerate}

\item $C_{\lambda}(X)$ is a $\sigma$-compact;

\item $X$ is a pseudocompact, $D(X)$ is a dense $C^*$-embedded set
in $X$ and family $\lambda$ consists of all finite subsets of
$D(X)$, where $D(X)$ is an isolated points of $X$.

\end{enumerate}

\end{theorem}

The closure of a set $A$ will be denoted by $\overline{A}$ (or
$cl(A)$); the symbol
 $\varnothing$ stands for the empty set. As usual, $f(A)$ and $f^{-1}(A)$ are the image and
 the complete preimage of the set $A$ under the mapping~$f$,
 respectively.

 A subset $A$ of a space $X$ is said to be {\it
 bounded} in $X$ if for every continuous function $f: X \mapsto
 \mathbb{R}$, $f\vert A: A\mapsto \mathbb{R}$ is a bounded
 function. Every $\sigma$-bounded space is projectively Menger
 (Proposition 1.1 in \cite{arh}).

\section{Main results}

In order to prove the main theorem we need to prove some
statements that we call Lemmas, but note their self-importance.

\medskip

Recall that a space $X$ is called basically disconnected
(\cite{gj}), if every cozero-set has an open closure. Clearly,
every basically disconnected (Tychonoff) space is zero-dimensional
space.

\begin{lemma}\label{lem2}
\label{lm3}

 If $C_{\lambda}(X)$ is Menger, then $X$ is a basically
disconnected space.
\end{lemma}

\begin{proof}

Let $U\subseteq X$ be a cozero set in $X$.
 Claim that  $\overline{U}=Int\overline{U}$. Suppose that
 $\overline{U}\setminus Int\overline{U}\neq \emptyset$.
Since $U$ is a cozero set, there are  open sets $U_n$ of $X$ such
that for each $n\in\mathbb N$,  $\overline{U_n}\subseteq U_{n+1}$
and $\bigcup\limits_{n=1}^{\infty}U_n=U$. For each $n,m\in\mathbb
N$, we put $Z_{n,m}=\{f\in C_{\lambda}(X,[0,1]): f\vert
(X\setminus Int\overline{U})\equiv 0$ and $f(U_n)\subset
[\frac{1}{2^m},1]\}$.

Note that $Z_{n,m}$ is closed subset of $C_{\lambda}(X)$ for each
$n,m\in\mathbb N$. Let $h\notin Z_{n,m}$.

If $x\in X\setminus Int\overline{U}$ such that $h(x)\neq 0$. Since
$\lambda$ is $\pi$-network of $X$, there is $A\in \lambda$ such
that $A\subset h^{-1}(h(x)-\frac{|h(x)|}{2},
h(x)+\frac{|h(x)|}{2})\bigcap Int (X\setminus Int\overline{U})$.
Then $h\in [A, (h(x)-\frac{|h(x)|}{2}, h(x)+\frac{|h(x)|}{2})]$
and $[A, (h(x)-\frac{|h(x)|}{2}, h(x)+\frac{|h(x)|}{2})]\bigcap
Z_{n,m}=\emptyset$.

If $x\in U_n$ and $h(x)\notin [\frac{1}{2^m},1]$. Let
$d=\frac{diam(h(x),[\frac{1}{2^m},1])}{2}$. Since $\lambda$ is a
$\pi$-network of $X$, there is $A\in \lambda$ such that $A\subset
h^{-1}((h(x)-d, h(x)+d)\bigcap U_n$. Then $h\in [A, (h(x)-d,
h(x)+d)]$ and $[A, (h(x)-d, h(x)+d)]\bigcap Z_{n,m}=\emptyset$.

 Assume that
$\bigcap \{Z_{n,m}: n\in \mathbb{N}\}=\emptyset$ for all $m\in
\mathbb{N}$. Using the Menger property of $C_{\lambda}(X)$, we can
take some $\varphi\in \mathbb{N}^{\mathbb{N}}$ such that $\bigcap
\{Z_{\varphi(m),m} :m\in \mathbb{N}\}=\emptyset$. For each $m\in
\mathbb{N}$, take any $g_m\in C_{\lambda}(X)$ satisfying
$g_m(X\setminus Int(\overline{U}))\equiv 0$ and
$g_m(U_{\varphi(m)})=\{1\}$. Let $g=\sum^{\infty}_{j=1}
2^{-j}g_j$. Then, $g\in C_{\lambda}(X)$ and $g(X\setminus
Int(\overline{U}))\equiv 0$. Fix any $m\in\mathbb{N}$, $1\leq k
\leq \varphi(m)$ and $x\in U_k$. Then we have

$g(x)=\sum^{\infty}_{j=1} 2^{-j}g_j(x)\geq 2^{-m}g_m(x)=2^{-m}$.

Hence, $g\in \bigcap \{Z_{\varphi(m),m} :m\in \mathbb{N}\}$. This
is a contradiction. Thus, there is some $m\in \mathbb{N}$ such
that $\bigcap \{Z_{n,m} : n\in \mathbb{N}\}\neq\emptyset$. Let
$p\in \bigcap \{Z_{n,m} : n\in \mathbb{N}\}$. Then $p(U)\subset
[\frac{1}{2^m},1]$ and $p\vert (X\setminus Int\overline{U})\equiv
0$. It follows that $\overline{U}\setminus Int\overline{U}=
\emptyset$.

\end{proof}

A subset $G\subset \omega^{\omega}$ is {\it dominating} if for
every $f\in \omega^{\omega}$ there is a $g\in G$ such that
$f(n)\leq g(n)$ for all but finitely many $n$.

\begin{theorem}(Hurewicz \cite{hur1})\label{hur}
A second countable space $X$ is Menger iff for every continuous
mapping $f: X\mapsto \mathbb{R}^{\omega}$, $f(X)$ is not
dominating.
\end{theorem}

"Second countable" can be extended to "Lindel$\ddot{o}$f":

\begin{theorem}(Ko$\check{c}$inac \cite{koc}, Theorem 2.2)\label{koc}
A Lindel$\ddot{o}$f space $X$ is Menger iff for every continuous
mapping $f: X\mapsto \mathbb{R}^{\omega}$, $f(X)$ is not
dominating.
\end{theorem}

\begin{lemma}\label{lem1}
\label{lm1}

If $C_{\lambda}(X)$ is Menger. Then $X$ is pseudocompact.
\end{lemma}

\begin{proof} Assume that $X$ is not pseudocompact and  $f\in C(X)$ is not bounded
function. Without loss of generality we can assume that
$\mathbb{N}\subset f(X)$. For each $n\in \mathbb{N}$ we choose
$A_n\in \lambda$ such that $A_n\subset f^{-1}((n-\frac{1}{3},
n+\frac{1}{3}))$. By Lemma \ref{lem2},
$F_n=\overline{f^{-1}((n-\frac{1}{3}, n+\frac{1}{3}))}$ is clopen
set for each $n\in \mathbb{N}$. Let $K=\{f\in C(X): f\vert
F_n\equiv s_{f,n}$ for each $n\in \mathbb{N}$ and $s_{f,n}\in
\mathbb{R} \}$. Then $K$ is closed subset of $C_{\lambda}(X)$ and,
hence, it is Menger. Fix $a_n\in A_n$ for every $n\in \mathbb{N}$.
Note that $D=\{a_n : n\in \mathbb{N}\}$ is a $C$-embedded copy of
$\mathbb{N}$ (3L (1) in \cite{gj}). So we have a continuous
mapping $F: K \mapsto \mathbb{R}^D$ the space $K$ onto
$\mathbb{R}^D$. But
$F(K)=\mathbb{R}^D=\mathbb{R}^{\mathbb{\omega}}$ is dominating,
contrary to the Theorem \ref{koc}.

\end{proof}

\begin{lemma}\label{lem3} If $C_{\lambda}(X)$ is Menger, then $\mu=\{A\in
\lambda:$ $A$ is finite subset of $X \}$ is a $\pi$-network of
$X$.

\end{lemma}

\begin{proof} Assume that there exist an open set $U$ of $X$ such
that $B\not\subset U$ for every $B\in \mu$. Fix a family $\{V_n:
n\in \mathbb{N}\}$ of open subsets of $X$ such that $V_n\subset U$
for every $n\in \mathbb{N}$ and $V_{n'}\bigcap V_{n''}=\emptyset$
for $n'\neq n''$. Fix $x_n\in V_n$ and $\epsilon>0$. For every
$f\in C_{\lambda}(X)$ and $n\in \mathbb{N}$ consider $B_{f,n}\in
\lambda$ such that $B_{f,n}\subset f^{-1}((f(x_n)-\epsilon,
f(x_n)+\epsilon))\bigcap V_n$. Then $\mathcal{U}_n=\{[B_{f,n},
(f(x_n)-\epsilon, f(x_n)+\epsilon)] : f\in C_{\lambda}(X)\}$ is an
open cover of $C_{\lambda}(X)$ for every $n\in \mathbb{N}$. Using
the Menger property of $C_{\lambda}(X)$, for sequence
$\{\mathcal{U}_n : n\in \omega\}$ of open covers of
$C_{\lambda}(X)$, there are finite subfamilies
$\mathcal{S}_n\subset \mathcal{U}_n$ such that $\bigcup
\{\mathcal{S}_n : n\in \omega \}$ is a cover of $C_{\lambda}(X)$.
Let $\mathcal{S}_n=\{[B_{f_1,n}, (f_{1,n}(x_n)-\epsilon,
f_{1,n}(x_n)+\epsilon)], ..., [B_{f_{k(n)},n},
(f_{k(n),n}(x_n)-\epsilon, f_{k(n),n}(x_n)+\epsilon)]\}$ for every
$n\in \mathbb{N}$. Since $B_{f_{s},n}$ is an infinite subset of
$X$, we fix $z_{s,n}\in B_{f_{s},n}$ for every $s\in
\overline{1,k(n)}$ and $n\in \mathbb{N}$ such that $z_{s',n}\neq
z_{s'',n}$ for $s'\neq s''$. Let $Z=\{z_{s,n}: s\in
\overline{1,k(n)}$ and $n\in \mathbb{N}\}$.

Define the function $q: Z\mapsto \mathbb{R}$ such that
$q(z_{s,n})=0$ if $0\notin (f_{s,n}(x_n)-\epsilon,
f_{s,n}(x_n)+\epsilon)$, else $q(z_{s,n})=2\epsilon$ for $s\in
\overline{1,k(n)}$ and $n\in \mathbb{N}$. By Lemma \ref{lem2}, $X$
is a basically disconnected space.

Recall that (14N p.215 in \cite{gj}) every countable set in a
basically disconnected space is $C^*$-embedded.

Hence, there is $t\in C_{\lambda}(X)$ such that $t\vert Z=q$. But
$t\notin \bigcup \{\mathcal{S}_n : n\in \omega \}$. This is a
contradiction.

\end{proof}

Denote $D(X)$ a set of isolated points of $X$.

\begin{lemma}\label{lem4} If $C_{\lambda}(X)$ is Menger, then $D(X)$ is dense
set in $X$.
\end{lemma}

\begin{proof} Assume that there exist an open set $W\neq\emptyset$ such that
$W\bigcap D(X)=\emptyset$. By Lemma \ref{lem3}, $\mu$ is
$\pi$-network of $X$, hence, there is $A\in \mu$ such that
$A\subset W$. Note that $X\setminus A$ is dense set in $X$. The
constant zero function defined on $X$ is denoted by $f_0$. For
every $f\in C(X)\setminus \{f_0\}$ there is $x_f\in X\setminus A$
such that $f(x_f)\neq 0$. For every $f\in C(X)\setminus \{f_0\}$,
consider $B_f\in \mu$ such that $B_f\subset
f^{-1}((f(x_f)-\frac{|f(x_f)|}{2},
f(x_f)+\frac{|f(x_f)|}{2}))\bigcap (X\setminus A)$. Let
$\epsilon>0$.  Then $\mathcal{V}=\{[B_f,
(f(x_f)-\frac{|f(x_f)|}{2}, f(x_f)+\frac{|f(x_f)|}{2})] : f\in
C(X)\setminus \{f_0\}\}\bigcup [A,(-\epsilon, \epsilon)]$ is an
open cover of $C_{\lambda}(X)$.  Since $C_{\lambda}(X)$ is Menger
and, hence, $C_{\lambda}(X)$ is Lindel$\ddot{o}$f, there is a
countable subcover $\mathcal{V}'=\{[B_{f_n},
(f_n(x_f)-\frac{|f_n(x_f)|}{2}, f_n(x_f)+\frac{|f_n(x_f)|}{2})]$ :
for $n\in \mathbb{N}\}\bigcup [A,(-\epsilon, \epsilon)]\subset
\mathcal{V}$ of $C_{\lambda}(X)$. Since $X$ is a basically
disconnected space and every countable set in a basically
disconnected space is $C^*$-embedded, there is $h\in C(X)$ such
that $h\vert \bigcup\limits_{n\in \mathbb{N}} B_{f_n}\equiv 0$ and
$h(a)=\epsilon$ for some $a\in A$. Note that $h\notin \bigcup
\mathcal{V}'$, to contradiction.

\end{proof}

\begin{lemma}\label{lem5} If $C_{\lambda}(X)$ is Menger, then $D(X)$ is $C^*$-embedded.
\end{lemma}

\begin{proof} Let $f$ be a bounded  continuous function from
$D(X)$ into $\mathbb{R}$, and $F_{A}=\{g\in C(X): g\vert A=f\vert
A\}$ for $A\in D(X)^{\omega}$. Note that $F_A$ is closed subset of
$C_{\lambda}(X)$ and, by Lemma \ref{lem2}, $F_A\neq\emptyset$. So
$\xi=\{F_A : A\in D(X)^{\omega}\}$ is family of closed subspaces
with the countable intersection property. Since $C_{\lambda}(X)$
is Menger, hence, it is Lindel$\ddot{o}$f, and every family of
closed subspaces of with the countable intersection property has
non-empty intersection. It follows that $\bigcap \xi\neq
\emptyset$. We thus get that $\widetilde{f}\in \bigcap \xi$ such
that $\widetilde{f}\in C(X)$ and $\widetilde{f}\vert D(X)=f$.

\end{proof}

\begin{proposition}\label{pr1} Let $X=\mathbb{N}$ and let
$\lambda=\{X\}\bigcup \{\{x\}: x\in X\}$. Then $C^*_{\lambda}(X)$
is not Menger.

\end{proposition}

\begin{proof} Assume that $C^*_{\lambda}(X)$ is Menger. For every $i\in \mathbb{N}$ consider an open cover
$\mathcal{V}_i=\{[\mathbb{N}, (-2+\frac{1}{i+1},
2-\frac{1}{i+1})]\}\bigcup \{[x,(-\infty, -2+
\frac{2i+1}{2i(i+1)})\bigcup (2-\frac{2i+1}{2i(i+1)}, +\infty)] :
x\in X\}$ of $C^*_{\lambda}(X)$. Using the Menger property of
$C^*_{\lambda}(X)$, for sequence $\{\mathcal{V}_i : i\in
\mathbb{N}\}$ of open covers of $C^*_{\lambda}(X)$, there are
finite subfamilies $\mathcal{S}_i\subset \mathcal{V}_i$ such that
$\bigcup \{\mathcal{S}_i : i\in \mathbb{N} \}$ is a cover of
$C^*_{\lambda}(X)$.

Without loss of generality we can assume that $[\mathbb{N},
(-2+\frac{1}{i+1}, 2-\frac{1}{i+1})]\in \mathcal{S}_i$ for each
$i\in \mathbb{N}$.

By using induction, for each $i\in \mathbb{N}$, determine the
values of the function $f$ at some points, depending on the
$\mathcal{S}_i$,  as follows:

 for $i=1$ and

 $\mathcal{S}_1=\{[\mathbb{N}, (-2+\frac{1}{2},
2-\frac{1}{2})], [x^1_1,(-\infty, -2+ \frac{3}{4})\bigcup
(2-\frac{3}{4}, +\infty)], ..., [x^1_k,(-\infty, -2+
\frac{3}{4})\bigcup (2-\frac{3}{4}, +\infty)]\}$, define

$f(x^1_n)=0$ for $n\in \overline{1,k}$ and

$f(s_1)=p_1$ where $p_1\in
[-2+\frac{5}{12},2-\frac{5}{12}]\setminus
(-2+\frac{1}{2},2-\frac{1}{2})$ for some $s_1\in X\setminus
\{x^1_n : n\in \overline{1,k}\}$. Denote $P_1=\bigcup\limits_{n\in
\overline{1,k}} x^1_n \bigcup s_1$.

for $i=m$

$\mathcal{S}_m=\{[\mathbb{N}, (-2+\frac{1}{m+1},
2-\frac{1}{m+1})], [x^m_1,(-\infty, -2+
\frac{2m+1}{2m(m+1)})\bigcup (2-\frac{2m+1}{2m(m+1)}, +\infty)],
...$

$..., [x^m_{k(m)},(-\infty, -2+ \frac{2m+1}{2m(m+1)})\bigcup
(2-\frac{2m+1}{2m(m+1)}, +\infty)]\}$, define

$f(x^m_n)=0$ where $x^m_n\notin P_{m-1}$ for $n\in
\overline{1,k(m)}$ and

$f(s_m)=p_m$ where $p_m\in [-2+
\frac{2(m+1)+1}{2(m+1)(m+2)},2-\frac{2(m+1)+1}{2(m+1)(m+2)}]\setminus
(-2+\frac{1}{m+1}, 2-\frac{1}{m+1})$ for some $s_m\in X\setminus
P_{m-1}$. Denote $P_m=\bigcup\limits_{n\in \overline{1,k(m)}}
x^m_n \bigcup s_m\bigcup P_{m-1}$ and $P=\bigcup\limits_{m\in
\mathbb{N}} P_m$.

If $X\setminus P\neq\emptyset$, then let $f(x)=1$ for $x\in
X\setminus P$.

By construction of $f$, $f\notin \mathcal{S}_i$ for every $i\in
\mathbb{N}$, to contradiction.

\end{proof}

\begin{lemma}\label{lem6}

If $C_{\lambda}(X)$ is Menger, then each $A\in \lambda$ is finite
subset of $D(X)$.

\end{lemma}

\begin{proof} Suppose that $C_{\lambda}(X)$ is Menger, $\widetilde{\lambda}=\{A\}\bigcup \{\{x\}, x\in D(X)\}$ and
$A\in \lambda$ is an infinite subset of $X$. Then
$C_{\widetilde{\lambda}}(X)$ is Menger, too. Note that if $A$ is
countable and $A\subset D(X)$, then we have a continuous mapping
$g: C_{\widetilde{\lambda}}(X)\mapsto C^*_{p\bigcup
\{\mathbb{N}\}}(\mathbb{N})$. Hence, $C^*_{p\bigcup
\{\mathbb{N}\}}(\mathbb{N})$ is Menger, contrary to Proposition
\ref{pr1}.

  Let $V=(-1,1)\bigcup(\mathbb{R}\setminus [-4,4])$. Consider
$\mathcal{U}=\{[A,V]\}\bigcup \{[x,\mathbb{R}\setminus
[-\frac{2}{3},\frac{2}{3}]]: x\in D(X)\}$. Since $D(X)$ is dense
subset of $C_{\widetilde{\lambda}}(X)$ (Lemma \ref{lem4}),
$\mathcal{U}$ is an open cover of $C_{\widetilde{\lambda}}(X)$
and, hence, there is a countable subcover $\mathcal{U}'\subset
\mathcal{U}$ of $C_{\widetilde{\lambda}}(X)$. Let
$\mathcal{U}'=\{[A,V],[x_1,\mathbb{R}\setminus
[-\frac{2}{3},\frac{2}{3}]], ..., [x_n,\mathbb{R}\setminus
[-\frac{2}{3},\frac{2}{3}]], ...\}$. Let $z\in A\setminus
\bigcup\limits_{n\in \mathbb{N}} \{x_n\}$ (note that either $z\in
A\setminus D(X)$ or $A\subset D(X)$ and $|A|>\aleph_0$). Since
every countable set in a basically disconnected space is
$C^*$-embedded, there is
 $h\in C_{\widetilde{\lambda}}(X)$ such that $h\vert \bigcup\limits_{n\in \mathbb{N}} \{x_n\}=0$ and
 $h(z)=2$. It follows that $h\notin \bigcup \mathcal{U}'$, to
 contradiction. It follows that  $A$ is finite
subset of $D(X)$.

\end{proof}

\begin{theorem}\label{th100} Let $X$ be a Tychonoff space and let $\lambda$ be a $\pi$-network of $X$. Then a space
$C_{\lambda}(X)$ is Menger, if and only if,  $C_{\lambda}(X)$ is
$\sigma$-compact.
\end{theorem}

\begin{proof} By Lemma \ref{lem1}, $X$ is
pseudocompact. By Lemmas \ref{lem3} and \ref{lem6}, the family
$\lambda$ consists of all finite subsets of $D(X)$, where $D(X)$
is an isolated points of $X$. By Lemma \ref{lem5}, $D(X)$ is a
dense $C^*$-embedded set in $X$. It follows that $C_{\lambda}(X)$
is $\sigma$-compact (Theorem \ref{th111}).

\end{proof}

 Various properties between $\sigma$-compactness and Menger are
 investigated in the papers \cite{tall,dtz}. We can summarize the relationships between considered
 notions in (\cite{tall}, see Figure 1),  Theorems \ref{th100} and
 \ref{th111}. Then we have the next

 \begin{theorem} For a Tychonoff space $X$ and  a $\pi$-network $\lambda$ of $X$, the following statements are
equivalent:

\begin{enumerate}

\item $C_{\lambda}(X)$ is $\sigma$-compact;

\item $C_{\lambda}(X)$ is Alster;

\item $(CH)$ $C_{\lambda}(X)$ is productively Lindel$\ddot{o}$f;

\item "TWO wins $M$-game" for $C_{\lambda}(X)$;

\item $C_{\lambda}(X)$ is projectively $\sigma$-compact and
Lindel$\ddot{o}$f;

\item $C_{\lambda}(X)$ is Hurewicz;

\item  $C_{\lambda}(X)$ is Menger;

\item $X$ is a pseudocompact, $D(X)$ is a dense $C^*$-embedded set
in $X$ and family $\lambda$ consists of all finite subsets of
$D(X)$, where $D(X)$ is an isolated points of $X$.

\end{enumerate}

 \end{theorem}

\section{Projectively Menger space}

According to Tka$\check{c}$uk \cite{tka}, a space $X$ said to be
{\it $b$-discrete} if every countable subset of $X$ is closed
(equivalently, closed and discrete) and $C^*$-embedded in $X$.

\medskip
\begin{lemma}(Lemma 2.1 in \cite{sa})\label{ls} The following are equivalent for a space
$X$:

\begin{enumerate}

\item  $X$ is $b$-discrete;

\item For any disjoint countable subsets $A$ and $B$ in $X$, there
are disjoint zero-sets $Z_A$ and $Z_B$ in $X$ such that $A\subset
Z_A$ and $B\subset Z_B$;

\item For any disjoint countably subsets $A$ and $B$ in $X$ such
that $A$ is closed in $X$, there are disjoint zero-sets $Z_A$ and
$Z_B$ in $X$ such that $A\subset Z_A$ and $B\subset Z_B$.

\end{enumerate}

\end{lemma}

\begin{definition} For $A\subset X$, a space $X$ will be called
{\it $b_A$-discrete } if every countable subset of $A$ is closed
in $A$ and $C^*$-embedded in $X$.
\end{definition}

\begin{lemma}\label{lsak} The following are equivalent for a space
$X$ and $A\subset X$:

\begin{enumerate}

\item  $X$ is $b_A$-discrete;

\item For any disjoint countable subsets $D$ and $B$ in $A$, there
are disjoint zero-sets $Z_D$ and $Z_B$ in $X$ such that $D\subset
Z_D$ and $B\subset Z_B$;

\item For any disjoint countably subsets $D$ and $B$ in $A$ such
that $D$ is closed in $A$, there are disjoint zero-sets $Z_A$ and
$Z_B$ in $X$ such that $D\subset Z_D$ and $B\subset Z_B$.

\end{enumerate}

\end{lemma}

Similarly to the proof of implication ($C_p(X,\mathbb{I})$ is
projectively Menger $\Rightarrow$ $X$ is $b$-discrete) of Theorem
2.4 in \cite{sa}, we claim the next

\begin{lemma}
\label{lm2}

 Let $C_{\lambda}(X)$ be a projectively Menger space, then $X$ is  $b_{A}$-discrete where $A=\bigcup \lambda$.
\end{lemma}

\begin{proof} Let $C_{\lambda}(X)$ be a projectively
Menger. We show the statement (3) in Lemma \ref{lsak}. Let $D$ and
$B$ be a disjoint countable subsets in $A$ such that $D$ is closed
in $A$. Let $B=\{b_n: n\in \mathbb{N}\}$, and let
$B_n=\{b_1,...,b_n\}$.

For each $n,m\in \mathbb{N}$, we put $Z_{n,m}=\{f\in
C_{\lambda}(X): f(D)=\{0\}$ and $f(B_m)\subset [\frac{1}{2^n},
1]\}$. Since $D$ and $B_m$ are countable and $\lambda$ is a
$\pi$-network of $X$, each $Z_{n,m}$ is a zero-set in
$C_{\lambda}(X)$. Assume that $\bigcap \{Z_{n,m}: m\in
\mathbb{N}\}=\emptyset$ for all $n\in \mathbb{N}$. Using the
projective Menger property of $C_{\lambda}(X)$, Theorem 6 in
\cite{bcm}, we can take some $\varphi\in \mathbb{N}^{\mathbb{N}}$
such that $\bigcap \{Z_{n,\varphi(n)} : n\in
\mathbb{N}\}=\emptyset$. For each $n\in \mathbb{N}$, take any
$g_n\in C_{\lambda}(X)$ satisfying $g_n(D)=\{0\}$ and
$g_n(B_{\varphi(n)}=\{1\}$. Let $g=\sum\limits_{j=1}^{\infty}
2^{-j}g_j$.  Then, $g\in C_{\lambda}(X)$ and $g(D)\equiv 0$. Fix
any $n\in\mathbb{N}$, $1\leq k \leq \varphi(m)$. Then we have

$g(b_k)=\sum^{\infty}_{j=1} 2^{-j}g_j(b_k)\geq
2^{-n}g_n(b_k)=2^{-n}$.

Hence, $g\in \bigcap \{Z_{n, \varphi(n)} : n\in \mathbb{N}\}$.
This is a contradiction. Thus, there is some $n\in \mathbb{N}$
such that $\bigcap \{Z_{n,m} : m\in \mathbb{N}\}\neq\emptyset$.
Let $h\in \bigcap \{Z_{n,m} : m\in \mathbb{N}\}$. Then $D\subset
Z_A=h^{-1}(0)$ and $B\subset Z_B=h^{-1}([\frac{1}{2^n},1])$.

\end{proof}

\begin{theorem}\label{th200} Let $X$ be a Tychonoff space and let $Y$ be a dense subset of $X$.
 Then the following statements are equivalent:

\begin{enumerate}

\item $C_{p}(Y\vert X)$ is projectively Menger;

\item $C_{p}(Y\vert X)$ is $\sigma$-bounded;

\item $C_{p}(Y\vert X)$ is $\sigma$-pseudocompact;

\item $X$ is pseudocompact and $b_{Y}$-discrete.

\end{enumerate}

\end{theorem}

\begin{proof}

Note that  $C_{p}(Y\vert X)$ is homeomorphic to $C_{\lambda}(X)$
for $\lambda=[Y]^{<\omega}$.

$(1)\Rightarrow (4)$. By Lemma \ref{lm2}, $X$ is $b_{Y}$-discrete.
Assume that $X$ is not pseudocompact and  $f\in C(X)$ is not
bounded function. Without loss of generality we can assume that
$\mathbb{N}\subset f(X)$. For each $n\in \mathbb{N}$ we choose
$a_n\in Y$ such that $a_n\in  f^{-1}((n-\frac{1}{3},
n+\frac{1}{3}))$.  Note that $D=\{a_n : n\in \mathbb{N}\}$ is a
$C$-embedded copy of $\mathbb{N}$ (3L (1) in \cite{gj}). So we
have a continuous mapping $F: C_p(Y\vert X) \mapsto \mathbb{R}^D$
the Menger space $C_p(Y\vert X)$ onto $\mathbb{R}^D$. But
$F(C_p(Y\vert X))=\mathbb{R}^D=\mathbb{R}^{\mathbb{\omega}}$ is
dominating, contrary to the Theorem \ref{koc}.

 $(4)\Rightarrow (3)$.  Since $C_{p}(Y\vert X,\mathbb{I})$ is a dense subset of
$\mathbb{I}^{Y}$ and $X$ is $b_{Y}$-discrete, by Proposition
\ref{pr13}, $C_{p}(Y\vert X,\mathbb{I})$ is pseudocompact. Hence,
$C_{p}(Y\vert X)$ is $\sigma$-pseudocompact.

Note that every $\sigma$-pseudocompact space is $\sigma$-bounded,
and every $\sigma$-bounded space is projectively Menger
 (Proposition 1.1 in \cite{arh}).
\end{proof}

\section{Examples}

Using Theorem \ref{th100} and Theorem \ref{th200}, we can
construct example of projective Menger topological group
$C_{\lambda}(X)$ such that it is not Menger.

Note that if $\lambda=[\bigcup \lambda]^{<\omega}$, then
$C_{\lambda}(X)$ is a topological group (locally convex
topological vector space, topological algebra) (\cite{os},
\cite{os1}).

\begin{example} (Example 1 in \cite{ospy})
Let $T$ be a $P$-space without isolated points, $X=\beta(T)$ and
let $\lambda$ be a family of all finite subsets of $T$. Then
$C_{\lambda}(X)$ is $\sigma$-countably compact (Theorem 1.2 in
\cite{ospy}), hence, the topological group $C_{\lambda}(X)$ is
projective Menger. But the space $X$ does not contain isolated
points, hence, $C_{\lambda}(X)$ is not Menger.
\end{example}

\begin{example}(Example 2 in \cite{ospy}) Let $D$ be an uncountable discrete space and $\lambda=D^{<\omega}$. Consider
$F=\beta(D)\setminus\bigcup\{\overline{S}: S\subset
D,\,\mbox{and}\, S\, \, \, \mbox{countable} \}$. Denote by $b(D)$
a quotient space obtained from $\beta(D)$ by identifying the set
$F$ with the point $\{F\}$.  Then  the topological group
$C_{\lambda}(b(D))$ is projective Menger ($\sigma$-countably
compact), but is not Menger.

\end{example}

\begin{example}(\cite{sha}) D.B.Shahmatov has constructed for an
arbitrary cardinal $\tau\geq 2^{\aleph_0}$ an everywhere dense
pseudocompact space $X_{\tau}$ in $\mathbb{I}^{\tau}$ such that
$X_{\tau}$ is a $b$-discrete. Hence, the topological group
$C_p(X_{\tau})$ is projective Menger ($\sigma$-pseudocompact and
is not $\sigma$-countably compact), but is not Menger for an
arbitrary cardinal $\tau\geq 2^{\aleph_0}$.

\end{example}

\begin{remark} By Theorems \ref{th111} and \ref{th100}, if $X$
is compact, $\lambda$ is a $\pi$-network of $X$ and
$C_{\lambda}(X)$ is Menger, then
 $X$ is homeomorphic to $\beta(D)$, where $\beta(D)$ is  Stone-$\check{C}$ech compactification
of a discrete space $D$, and  $\lambda=[D]^{<\omega}$.
\end{remark}
\medskip

\bibliographystyle{model1a-num-names}
\bibliography{<your-bib-database>}

\begin{thebibliography}{10}


\bibitem{arkh}
A.V. Arhangel'skii, \textit{Continuous maps, factorization
theorems, and function spaces}, Trudy Moskovsk. Mat. Obshch., 47,
(1984), 3--21.


\bibitem{arh2}
A.V. Arhangel'skii, \textit{Hurewicz spaces, analytic sets and fan
tightness of function spaces}, Sov. Math. Dokl., 33, (1986),
396--399.

\bibitem{arh}
A.V. Arhangel'skii, \textit{Projective $\sigma$-compactness,
$\omega_1$-caliber, and $C_p$-spaces}, Topology and its
Applications, 157, (2000), 874--893.







\bibitem{arp}
A.V. Arhangel'skii, V.I. Ponomarev, \textit{Fundamentals of
general topology: problems and excercises}, Reidel, 1984.
(Translated from the Russian.)







\bibitem{bcm}
M. Bonanzinga, F. Cammaroto, M. Matveev, \textit{Projective
versions of selection principles}, Topology and its Applications,
157, (2010), 874--893.


\bibitem{dtz}
H. Duanmu, F.D. Tall, L. Zdomskyy, \textit{Productively
Lindel$\ddot{o}$f and indestructibly Lindel$\ddot{o}$f spaces},
Topology and its Applications 160:18 (2013), 2443-2453.






\bibitem{enge}
R. Engelking, \textit{General Topology}, PWN, Warsaw, (1977); Mir,
Moscow, (1986).





\bibitem{gj}
L. Gillman, M. Jerison, \textit{Rings of continuous functions},
 The University Series in Higher Mathematics. Princeton, New
Jersey: D. Van Nostrand Co., Inc., 1960. 300~p.



\bibitem{hur}
W. Hurewicz, \textit{$\ddot{U}$ber eine verallgemeinerung des
Borelschen Theorems}, Math. Z. 24 (1925) 401-421.

\bibitem{hur1}
W. Hurewicz, \textit{$\ddot{U}$ber folger stetiger funktionen},
Fund. Math. 9 (1927) 193-204.


\bibitem{koc}
Lj.D.R. Ko$\check{c}$inac, \textit{Selection principles and
continuous images}, Cubo Math.J. 8 (2) (2006) 23--31.




\bibitem{nokh}
S.E. Nokhrin, \textit{Some properties of set-open topologies},
Jurnal of Mathematical Sciences, issue 144, n~3, (2007)
4123--4151.


\bibitem{ospy}
A.V. Osipov, E.G. Pytkeev, \textit{On the $\sigma$-countable
compactness of spaces of continuous functions with the set-open
topology}, Proceedings of the Steklov Institute of Mathematics,
issue 285, n.~S1, (2014) 153--162.



\bibitem{os}
A.V. Osipov, \textit{Topological-algebraic properties of function
spaces with set-open topologies},  Topology and its Applications,
issue 3, n.~159, (2012) 800--805.


\bibitem{os1}
A.V. Osipov, \textit{Group structures of a function spaces with
the set-open topology}, Sib. $\grave{E}$lektron. Mat. Izv., 14,
(2017) 1440-1446.




\bibitem{sa}
M.~Sakai, \textit{The projective Menger property and an embedding
of $S_{\omega}$ into function spaces}, Topology and its
Applications, Vol. 220 (2017) 118--130.

\bibitem{sash}
M. Sakai, M. Scheepers, \textit{The combinatorics of open covers}
in: K.P. Hart, J. van Mill, P.Simon (Eds.), Recent Progress in
General Topology III, Atlantic Press, 2014, pp. 751--799.







\bibitem{sha}
D.B. Shahmatov, \textit{A pseudocompact Tychonoff space all
countable subsets of which are closed and $C^*$-embedded},
Topology and its Applications, 22:2, (1986), 139--144.

\bibitem{tall}
F.D. Tall, \textit{Productively Lindel$\ddot{o}$ff spaces may all
be $D$}, Canadian Mathematical Bulletin 56:1 (2013), 203--212.




\bibitem{tka}
V.V. Tka$\check{c}$uk, \textit{The spaces $C_{p}(X)$:
decomposition into a countable union of bounded subspaces and
completeness properties}, Topology and its Applications, n~22,
(1986), 241--253.





\end{thebibliography}







\end{document}